\documentclass{amsart}
\usepackage{amsmath,amssymb,enumerate}
\usepackage{mathtools}
\usepackage{soul}

\usepackage{epsfig}
\usepackage[dvipsnames]{xcolor}

\newtheorem{theorem}{Theorem}[section]

\newtheorem{lemma}[theorem]{Lemma}

\newtheorem{definition}[theorem]{Definition}

\newtheorem{question}[theorem]{Question}

\newcommand{\occult}[1]{}

\newcommand\ko[1]{\marginpar{\tiny \ko{#1}}}

\usepackage{soul}

\newcommand{\Reg}{\mathcal{R}}     %{\mathrm{Reg}}
\newcommand{\Sing}{\mathcal{S}}     %{\mathrm{Sing}}

\begin{document}

\title[Phase Transitions]{Phase transitions for the geodesic flow of a rank one surface with nonpositive curvature}

\begin{abstract} 
 We study the one parameter family of potential functions $q\varphi^u$ associated with the geometric potential $\varphi^u$ for the geodesic flow of a compact rank 1 surface of nonpositive curvature.  For $q<1$ it is known that there is a unique equilibrium state associated with $q\varphi^u$, and it has full support. For $q > 1$ it is known that an invariant measure is an equilibrium state if and only if it is supported on the singular set. We study the critical value $q=1$ and show that the ergodic equilibrium states are either the restriction to the regular set of the Liouville measure, or measures supported on the singular set. In particular, when~$q = 1$, there is a unique ergodic equilibrium state that gives positive measure to the regular set.
\end{abstract}

\author{K.~Burns, J.~Buzzi, T.~Fisher, and N.~Sawyer}

\address{K.~Burns, Department of Mathematics, Northwestern University, Evanston, IL 60208, USA, \emph{E-mail address:} \tt{burns@math.northwestern.edu}}
\address{J.~Buzzi, Laboratoire de Mathématiques d'Orsay, CNRS - UMR 8628
Universit\'e Paris-Sud 11, 91405 Orsay, France, \emph{E-mail address:}
\tt{jerome.buzzi@math.u-psud.fr}}
\address{T.~Fisher, Department of Mathematics, Brigham Young University, Provo, UT 84602, USA, \emph{E-mail address:} \tt{tfisher@mathematics.byu.edu}}
\address{N.~Sawyer, Department of Mathematics and Computer Science, Southwestern University, Georgetown, TX 78626, USA, \emph{E-mail address:} \tt{nsawyer@southwestern.edu}}

\thanks{T.F.\ is supported by Simons Foundation Grant \# 239708.  }
\thanks{J.B.\ is supported by the ANR grant ISDEEC \# ANR-16-CE40-0013.  }


\maketitle

\newcommand\blue{\color{blue}}

\section{Introduction}

{
%\marginpar{\tiny informal explanation of the setting, result and proof of this paper}
Thermodynamical formalism brings ideas from statistical mechanics into ergodic theory  \cite{RuelleBook, Sinai-Gibbs-Ergodic}, and builds  natural measures such as the Sina\u\i-Ruelle-Bowen measures or the maximal entropy measures as \emph{equilibrium states} with respect to appropriate \emph{potential functions}. In the classical setting of uniformly hyperbolic dynamics (e.g., geodesic flows on closed manifolds of negative curvature) every H\"older continuous potential function defines a unique equilibrium measure. This equilibrium state can be characterized in several ways. It is the limit of uniform measures on the periodic orbits with weights determined by the potential. Also the integral of a continuous function with respect to the equilibrium state is the forward Birkhoff average of the function at a Lebesgue typical point.

%is the unique measure such that describe Lebesgue almost every orbits or the distribution of periodic orbits \cite{rB75}. 

%\marginpar{\tiny nonuniformly hyperbolic dynamics, rank one, phase transitions}
Since the 1980s, this formalism has been studied in various nonuniformly hyperbolic settings. The geodesic flows for  closed \emph{rank one manifolds} with nonpositive curvature \cite{Ballmann82} have been key examples. New phenomena appear such as \emph{phase transitions}: for special potential functions, the uniqueness of the equilibrium state fails. A physical analogue is the coexistence of ice and liquid water at a precise temperature for a given pressure \cite{physics-book,Sarig-Critical}.

%\marginpar{\tiny our result and proof}
In this paper we consider the geodesic flow for a closed rank one surface with nonempty singular set and the family of potentials $q\varphi^u$, where $\varphi^u$ is the  \emph{unstable Jacobian potential}, which we define below, and $q \in \mathbb R$. It is known (see for example the introduction to \cite{BG}) that this family has a phase transition at $q = 1$. Using the Ledrappier-Young entropy theory, we show that at the phase transition, when $q = 1$, the only equilibrium state not carried by the singular subset is the regular Liouville measure restricted to the regular set.

\section{Precise statement and previous results} 
    
{
We consider the dynamical systems defined by} a flow $\mathcal F =(f^t)_{t\in\mathbb{R}}$ on a manifold  and  a potential $\varphi$, which is a continuous real valued function on the manifold. The topological pressure $P(\varphi)$ is defined as
  $$
  P(\varphi)=\sup_{\mu}h_\mu(\mathcal{F}) + \int \varphi \,d\mu,
  $$ 
  where the supremum is over all $\mathcal{F}$-invariant Borel probability measures.  An \emph{equilibrium state} is a measure that achieves the supremum.  By a theorem of Newhouse \cite{Newhouse1989} (see \cite{Buzzi1997} for a simplified proof) we know that equilibrium states always exist if  $\mathcal{F}$ is $C^\infty$. We will consider equilibrium states for the one parameter family of potentials $q\varphi$ for $q \in \mathbb R$.

This paper studies a specific class of geometric examples. The flow $\mathcal{F}$ is the geodesic flow on the unit tangent bundle $T^1M$ of a
 compact connected $C^\infty$ Riemannian surface $M = (M,g)$. 
We assume that the Gaussian curvature $K$ of the metric $g$ is  everywhere nonpositive and somewhere negative. Such surfaces are known as rank one surfaces of nonpositive curvature; this is the two dimensional case of  the  notion of a rank one manifold of nonpositive curvature introduced in \cite{BBE}.
As a consequence of the Gauss-Bonnet theorem,  $M$ has genus at least $2$. 

In this setting, there are two continuous invariant subbundles $E^s$ and $E^u$ of $TT^1M$. Each of these is one dimensional and  orthogonal to the flow direction $E^c$ in the natural Sasaki metric.  As we explain in the next section, $E^s$ corresponds to the orthogonal Jacobi fields whose length is nonincreasing for all time and  $E^u$  to the orthogonal Jacobi fields whose length is nondecreasing for all time. A vector $v \in T^1M$ is said to be \emph{regular} if $E^s(v) \cap E^u(v)  = \{0\}$. Otherwise $v$ is said to be \emph{singular}. If $v$ is singular, $E^s(v) = E^u(v)$ and there are nonzero covariantly constant orthogonal Jacobi fields along the geodesic $\gamma_v$ that has initial tangent vector $v$. It is easily seen that $v$ is singular if and only if $K(\gamma_v(t)) = 0$ for all $t$.
We denote by $\Reg$ the set of all regular vectors and by $\Sing$ the set of all singular vectors.  The set $\Reg$ is obviously open and invariant while $\Sing$ is closed and invariant.
%\tf{Do we prefer $\mathcal{R}$ or $\mathrm{Reg}$ for the regular set?  Similarly for the singular set?}

There is a natural smooth measure on $T^1M$ known as the Liouville measure. It is invariant under the geodesic flow. The set $\Reg$ was shown to be dense by Ballmann \cite{Ballmann82}. 
It has full Liouville measure in all known examples, but whether this is always true {is a major open question. This set $\Reg$ however always has positive Lebesgue measure so one can restrict the Liouville measure to $\Reg$ and then renormalize it} to be a probability measure, which will be denoted by $\mu_L$.

The potential function that we study  is the \emph{geometric potential} 
 $$\varphi^u(v)=-\lim_{t\to 0}\frac{1}{t} \log \det (df^t|_{E^u_v}).$$
 This potential has been studied in the context of hyperbolic systems since the pioneering work of Bowen.  
%\marginpar{\tiny\blue Is it true that $\varphi^u$ detects curvature along the whole orbit of $v$? Reference?}
 In our setting, $\varphi^u = 0$ on $\Sing$ and $\varphi^u(v) < 0$ for any $v \in \Reg$ such that $K(\gamma_v(t)) < 0$  for some $t < 0$. In particular, $\varphi^u(v) < 0$ if $v$ is generic for any invariant probability measure $\mu$ such that $\mu(\Reg) = 1$; see Lemma \ref{l:hyp}. 
 
 We consider the one parameter family of potentials $q\varphi^u$ where $q$ ranges over $\mathbb R$.

\begin{definition}
For a potential $\varphi$, a \emph{phase transition} occurs at  $q_0$ if the map  $q\mapsto P(q\varphi)$ is not differentiable at $q_0$. 
%\footnote{Lack of differentiability implies nonuniqueness of the equilibrium states by a simple convexity argument, see \cite{Ruelle78}.}% there is more than one equilibrium state for $q_0\varphi$.
\end{definition}

Note that nondifferentiability implies nonuniqueness of the equilibrium states by a simple convexity argument; see \cite{Ruelle78}.

 If $\Sing = \emptyset$, the geodesic flow is transitive and Anosov \cite{Eberlein73}. Any multiple of the geometric potential $\varphi^u$ is  H\"older continuous, and it is a classical result of Bowen \cite{Bow75} that all H\"older continuous potentials have unique equilibrium states when the flow is transitive and Anosov. All of these equilibrium states are Bernoulli, fully supported and the weak$^*$ limit of weighted measures supported on periodic closed orbits.
So there are no phase transitions in this case and we henceforth assume that $\Sing \neq \emptyset$.  Bowen's results hold in any dimension.

%\tf{If we use this definition then we need to establish that the lack of differentiability}

%We note that in the physical literature \cite{physics-book}, phase transitions (or critical phenomena) includes additional aspects:  such as the ``failure of the Central Limit Theorem" for the equilibrium states at $q=q_0$. We refer to \cite{Sarig-Critical} for a corresponding mathematical picture in some large class of symbolic dynamical systems (Markov shifts with countable state alphabets).

If $\Sing \neq \emptyset$ and $M$ is a surface, then the geodesic flow lies on the boundary of the Anosov flows. Phase transitions {are known to} arise for {some systems similarly} on the boundary of hyperbolic systems.  For instance, the partially hyperbolic maps introduced in \cite{DHRS} that are on the boundary of uniformly hyperbolic diffeomorphisms have a phase transition as shown by Leplaideur, Oliveira, and Rios \cite{LOR} for a family of potential functions related to the {geometric potential}.  Specifically, for a partially hyperbolic diffeomorphism $F$ with center distribution $E^c$ they show that the family of potential functions $\phi_q=q\log |DF|_{E^c}|$ has a phase transition for a specific parameter $q_0$.  For certain partially hyperbolic, topologically transitive, local diffeomorphisms there are phase transitions \cite{DGR} where each ergodic measure has positive entropy.
%\tf{Other papers to cite that have previous results?}

When $\Sing \neq \emptyset$ the first major result  was by Knieper who used a combination of ideas from Patterson-Sullivan and Bowen  to prove uniqueness of the measure of maximal entropy \cite{knieper98}.
%\footnote{His result also applies to rank one manifolds of nonpositive curvature in higher dimensions.}.
Climenhaga and Thompson \cite{CT3} developed a nonuniform version of Bowen's techniques, and this was used in \cite{BCFT} to show that $q\varphi^u$ has a unique equilibrium state for any $q$ sufficiently close to 0.  As in the Anosov case, these equilibrium states are Bernoulli, fully supported and the weak$^*$ limit of weighted measures supported on periodic closed orbits. The Bernoulli property is a consequence of work of Ledrappier, Lima and Sarig \cite{LLS}.

Stronger results hold for surfaces with $\Sing \neq \emptyset$.  In this case, Theorem C of \cite{BCFT} shows there is a unique equilibrium state for $q\varphi^u$ whenever $q<1$;  these measures have all of the properties described in the previous paragraph.

The results in \cite{BCFT} also apply to some other potential functions.  Climenhaga and Thompson's techniques have also been used to show there is a unique measure of maximal entropy for geodesic flows on certain manifolds without conjugate points \cite{CKW}, and to establish the uniqueness of equilibrium states for certain diffeomorphisms with weak forms of hyperbolicity \cite{CFT2, CFT}.

The pressure of $q\varphi^u$  was also studied in \cite{BG} as part of an investigation of the Hausdorff dimensions of the level sets of the Lyapunov exponent associated to $E^u$. In that paper it was noted that there is a phase transition at $q = 1$.
The graph of $P(q\varphi^u)$ is shown in Figure~1, which is essentially the same as Figure 2a) in \cite{BG}.

\begin{figure}[htb]
\begin{center}
\includegraphics[width=.6\textwidth]{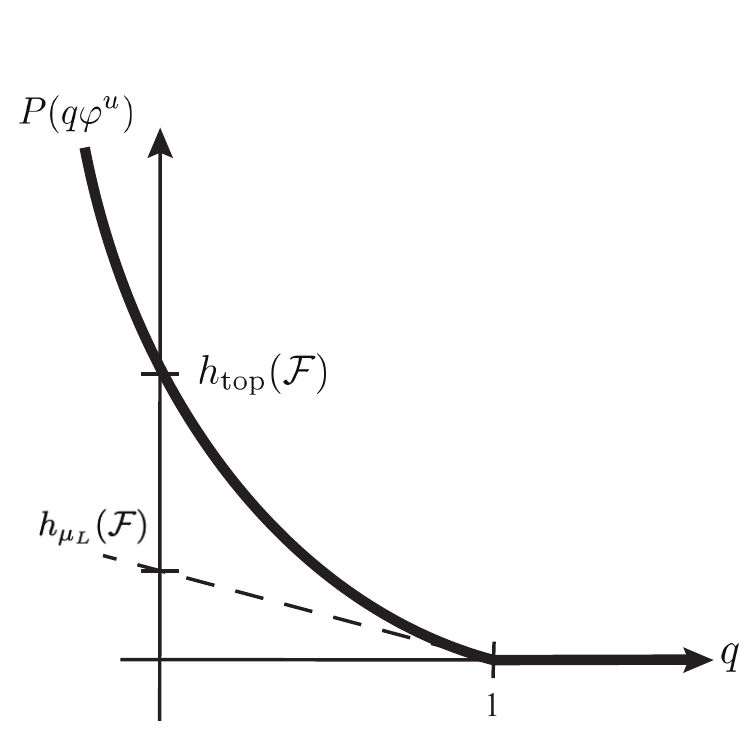}
\caption{Graph of the pressure}\label{f.graph}
\end{center}
\end{figure}
%\tf{modify figure}

%\marginpar{\blue\tiny slight reorganization of the following paragraphs until Theorem 2.2 - all arguments deferred to Section 3}
As we explain in Section 3, {the following facts follow} from the variational principle,  Ruelle's inequality and the Pesin entropy formula. First 
  \begin{equation} \label{Liouineq}
  P(q\varphi^u) \geq h_{\mu_L}(\mathcal F) + q \int \varphi^u \,d\mu_L,
  \end{equation}
and both sides are equal to $0$ when $q = 1$.

For $q>1$, we have $P(q\varphi^u) = 0$ which is trivially $C^1$ on $(1,\infty)$ and the equilibrium states for $q\varphi^u$ are the invariant probability measures supported on {the compact invariant set} $\Sing$. 

For $q<1$, Theorem C of \cite{BCFT} shows that $P(q\varphi^u)$ is $C^1$ on $(-\infty, 1)$ and that there is a unique equilibrium state $\mu_q$ for $q\varphi^u$. Moreover, this equilibrium state is ergodic since it is unique and satisifies $\mu_q(\Reg) = 1$, since the regular set has positive measure with respect to $\mu_q$ and the regular set is invariant.

For $q=1$, since the integral on the right hand side of (\ref{Liouineq}) is positive, there must be a discontinuity in the derivative at $q = 1$. Hence, there is a phase transition at $q=1$.  Both the invariant probability measures supported on $\Sing$ and $\mu_L$ are equilibrium states for $\varphi^u$.  The main result of this paper is that there are no other equilibrium states for this value $q=1$:

 \begin{theorem}\label{thm:phasetransition}
 For $\varphi^u$ the only ergodic equilibrium state not supported on $\Sing$ for the geodesic flow of a compact rank 1 surface of nonpositive curvature is $\mu_L$.
 \end{theorem}

%{This result does not extend to higher dimension. Indeed,}
%in dimension~3 there is a rank~1 manifold with nonpositive sectional curvature and nonempty singular set where there is a unique equilibrium state supported on $\Reg$ for $q\varphi^u$ for all $q\in\mathbb{R}$; see Section 10.2 of \cite{BCFT}.  

%\marginpar{\tiny\blue Does that Theorem D allow equilibrium states on $\Sing$, i.e., phase transitions, for small $q$?}

We note that little is known about phase transitions for the geodesic flow of compact rank one manifolds with nonpositive sectional curvature  in higher dimension. The Heintze example \cite[Example 4]{BBE} shows that at least two different  behaviours are possible. Heintze's example is constructed by taking two copies of a compact hyperbolic manifold of dimension $n \geq 3$ with finite volume and one cusp, cutting off the ends of the cusps, and gluing the two halves together. The resulting manifold has nonpositive curvature and  contains a flat totally geodesically isometrically embedded torus $T$ of dimension $n-1$.  Moreover, it can be arranged that each plane spanned by a  vector orthogonal to $T$ and a vector tangent to $T$ has sectional curvature zero or that all such planes have negative curvature.
%\TF{For each $p\in T$ we look at a plane $P$ in the tangent space of $p$ where $P$ is spanned by a vector normal to the $T$ and a vector tangent to $T$.}
%The construction can be made so that \TF{each plane has} \deleted{planes orthogonal to $T$ have} zero curvature or so that they \TF{each} has negative curvature. 
In the first case, we see the same phase transition at $q =1$ as in the surface case. In the second case, as explained in \cite[Section 10.2]{BCFT}, it can be arranged that   $q\varphi^u$ has a unique equilibrium state for each $q\in \mathbb{R}$ and there are no phase transitions.  On the other hand, Theorem D of \cite{BCFT} proves that the family $q\varphi^u$ does not have a phase transition at $q = 0$ for the geodesic flow of any compact rank 1 manifold with nonpositive sectional curvature.

%\TF{We note that little is known about phase transitions for the geodesic flow of compact rank 1 manifolds with nonpositive sectional curvature  in higher dimension.  The example in \cite[Section 10.2]{BCFT} shows $q\varphi^u$ has a unique equilibrium state for each $q\in \mathbb{R}$, and so there are no phase transitions for this system.  Theorem D of \cite{BCFT} proves that there is no phase transition at $q=0$ for $q\varphi^u$ and any geodesic flow of a compact rank 1 manifold with nonpositive sectional curvature.}

%for a compact rank 1 manifold with nonpositive sectional curvature there exists a $q_0>0$ such that $q\varphi^u$ has a unique equilibrium state supported on $\Reg$ for all $q\in (-q_0, q_0)$, but no lower bound for  $q_0$ is known (other than $0$).  This leads us to the following open questions.

%\marginpar{\tiny\blue do we ask about phase transition, ie, no equilibrium even in $\Sing$?}

\begin{question}\label{q:interval}
For any compact rank one manifold with nonpositive sectional curvature is there a unique equilibrium state associated with $q\varphi^u$ for each $0<q<1$?  What about for $q<0$?
\end{question}

%\tf{Should we mention that it seems likely the answer to the first question is negative?}

\begin{question}\label{q:phasetransition}
For a compact rank one manifold with nonpositive sectional curvature, what conditions guarantee that a phase transition occurs at some $q'\in\mathbb{R}$ for the  one parameter family of potential functions $q\varphi^u$?
\end{question}

\section{Jacobi fields and the derivative of the geodesic flow}

We outline some facts about the geodesic flow for manifolds with nonpositive curvature. A convenient source for the details is  \cite{Ebe:96}. 
%\tf{This the same reference as Ebe01 in BCFT.}
If $M$ is a Riemannian manifold and $\pi:TM \to M$ is the natural projection, we have for each $v \in TM$ a natural isomorphism $I_v: T_vTM \to T_{\pi(v)}M \times T_{\pi(v)}M$ defined as follows. Suppose $V(t)$ is a $C^1$ curve in $TM$
with footpoint curve $c(t)$ and $V'(t)$ is the covariant derivative of $V$ along $c$.
Then
  $$
  I_{V(0)}(\dot V(0)) = (\dot c(0), V'(0)).
  $$
We have $I_v(T_vT^1M) = T_{\pi(v)}M \times v^\perp$, where $v^\perp$ is the orthogonal complement of $v$ in~$T_{\pi(v)}M$. The vector field that generates the geodesic flow has value $I_v^{-1}(v,0)$ at~$v$. If $\xi = I_v^{-1}(w,w') \in T_vTM$, and $J_\xi$ is the Jacobi field along the geodesic $\gamma_v$ with $J(0) = w$ and $J'(0) = w'$, then 
 $$
 I_{F^t(v)}(DF^t(\xi)) = (J_\xi(t),J'_\xi(t))
 $$ for all $t$. The \emph{Sasaki metric} is defined on $TTM$ by
  $$
   \langle \xi, \xi\rangle = \langle w, w\rangle + \langle w', w'\rangle
    \quad\text{ if $\xi = I_v(w,w')$.}
  $$

A Jacobi field $J(t)$ along the  geodesic $\gamma_v$ is \emph{stable} (resp.\ \emph{unstable}) if $J(t) \perp \dot\gamma(t)$ for all $t$ and $\|J(t)\|$ is nonincreasing (resp.\ nondecreasing) for all $t$. A stable or unstable Jacobi field can vanish only if it is identically zero. 
If $w$ is a vector perpendicular to $v$, there are a unique stable Jacobi field and a unique unstable Jacobi field along $\gamma_v$ with initial value $w$. 
%\jb{doesn't this need $v$ to be regular?}
%\tf{this doesn't require it to be regular.}
The spaces of stable and unstable Jacobi fields along $\gamma_v$ are therefore both one dimensional. They also vary continuously. 
Consequently there are continuous functions $a^s,a^u: TT^1M\to \mathbb R$ such that if $w \in v^\perp$ then the stable (resp.\ unstable) Jacobi field along $\gamma_v$ with initial value $w$ has derivative $a^s(v)w$ (resp.\ $a^u(v)w$) at time $0$. It is clear that  $a^u \geq 0 \geq a^s$.

We set
  $$
  E^s(v) = \{\xi \in T_vT^1M \text{ such that $J_\xi$ is stable}\}
  $$
and 
$$
  E^u(v) = \{\xi \in T_vT^1M \text{ such that $J_\xi$ is unstable}\}.
  $$

It follows from the  Jacobi equation that $\langle J,J \rangle$ is a convex function for any Jacobi field $J$ along any geodesic $\gamma$. Moreover $\langle J,J\rangle''(t) > 0$ unless $J(t) = 0$ or  $K(\gamma(t)) = 0$.  It follows that stable and unstable Jacobi fields must have constant length and be covariantly constant along singular geodesics because the curvature is always zero. If $J$ is a nonzero unstable (resp.\ stable) Jacobi field and $K(\gamma(t_0)) < 0$, then $\|J\|'(t) > 0$ for all $t > t_0$ (resp.\ 
$\|J\|'(t) < 0$ for all $t < t_0$). We see that  $\|J\|$ cannot be constant if $J$ is a stable or unstable Jacobi field along a regular geodesic. It is also clear that
$a^u(v) = 0 = a^s(v)$ if $v$ is singular and $a^u(v) > 0 > a^s(v)$ if the curvature is negative at $\pi(v)$.

  The functions $a^s$ and $a^u$ are bounded because they are continuous. It follows easily that if $\xi \in E^s(v)$ or $\xi \in E^u(v)$ and $\lim_{|t| \to \infty}\frac1t\log\|Df^t(\xi)\|$ exists, then it is the same as $\lim_{|t| \to \infty}\frac1t\log\|J_\xi(t)\|$. From this it is easy to see that all three Lyapunov exponents for the geodesic flow $\mathcal F$ are $0$ on the singular set, while on the regular set we have the zero exponent associated to the flow direction and the exponents
$\chi^u$ and $\chi^s$ associated to the bundles $E^u$ and $E^s$. It is obvious from the above discussion that
if $v$ is regular and has all three Lyapunov exponents defined, then  $\chi^u(v) \geq 0 \geq \chi^s(v)$.

\begin{lemma}\label {l:hyp}
Let $\mu$ be an ergodic probability measure such that $\mu(\Reg) = 1$. Then
$\chi^u(v) > 0 > \chi^s(v)$ for $\mu$-a.e.~$v$.
\end{lemma}

\begin{proof} We must have 
  $$\int a^u(v)\,d\mu(v) > 0 .
 $$
Otherwise, since $a^u \geq 0$, we would know that $\mu$-a.e.~$v \in \Reg$ has $a^u(f^t(v))=0$ for almost all $t$. Since $a^u$ is continuous, this would mean that  for any such $v$ all unstable Jacobi fields along $\gamma_v$ would have constant length, which is impossible if $v \in \Reg$.

We now know that if $v$ is a generic point for $\mu$, the Birkhoff average of $a^u$ is positive at $v$. It is immediate from this that unstable Jacobi fields along $\gamma_v$ grow exponentially and therefore $\chi^u(v) > 0$.

A  similar argument applies to $\chi^s$.
\end{proof}

It is clear from their definitions that the exponent $\chi^u$ is the Birkhoff average of the function $-\varphi^u$. Finally, even though we do not use this fact, it is worth noting that  $\chi^s = -\chi^u$ because the flow $\mathcal F$ preserves the contact form on $TT^1M$.

\section{Proof of Theorem \ref{thm:phasetransition}}

The probability measures that are ergodic for $\mathcal F$ can be divided into two classes: those like $\mu_L$ for which $\Reg$ has measure $1$ and $\Sing$ has measure $0$ and those for which $\Sing$ has measure $1$ and $\Reg$ has measure $0$. The latter class is also nonempty because $\Sing$ is compact and invariant.

From Ruelle's inequality \cite{Ruelle78} and the fact that $\chi^u$ is the only Lyapunov exponent that can be positive,  we know that if $\mu$ is a probability measure that is ergodic for $\mathcal F$, then
$$
h_\mu(\mathcal{F})\leq \int \chi^u(v)\,d\mu(v).
$$
Pesin's entropy formula \cite{Pesin77} tells us that we have equality if $\mu = \mu_L$. We also have equality if $\mu$ is supported on $\Sing$, since the right hand side is $0$ in that case.  So  for any measure $\mu$ supported on the singular set we have
\begin{itemize}
\item $h_\mu(\mathcal{F})=0$,  
\item $\int \varphi^u(v)d\mu(v)=0$, and  
\item $h_\mu(\mathcal{F})-\int q\varphi^u(v)d\mu(v)=0$ for any $q\in \mathbb{R}$.
\end{itemize}   
Since $\chi^u$ is the Birkhoff average of the function $-\varphi^u$, we have

  \begin{equation}\label{atq=1}
  h_\mu(\mathcal{F}) +  \int \varphi^u(v)\,d\mu(v) \leq 0,
  \end{equation}
  with equality if $\mu = \mu_L$ or $\mu$ is supported on the singular set. This tells us that $\mu_L$ and measures supported on $\Sing$ are equilibrium states for $q\varphi^u$ when $q = 1$.

The integral in (\ref{atq=1}) is negative if $\mu(\Reg) =1$ and $0$ if if $\mu(\Sing) =1$.  If $q > 1$, we have
 \begin{equation*}\label{whenq>1}
  h_\mu(\mathcal{F}) +  \int q\varphi^u(v)\,d\mu(v) \leq 0,
  \end{equation*}
  with equality if and only if $\mu$ is supported on the singular set. This shows that when $q > 1$, the equilibrium states for $q\varphi^u$ are precisely the measures supported on the singular set.

If $\mu = \mu_L$, inequality (\ref{atq=1}) is an equality and the integral on the left hand side  is negative. Hence
 \begin{equation*}\label{whenq<1}\begin{aligned}
  P(q\varphi^u) &\geq h_{\mu_L}(\mathcal{F}) +  \int q\varphi^u(v)\,d\mu_L(v) \\
  &= h_{\mu_L}(\mathcal{F}) +  \int \varphi^u(v)\,d\mu_L(v) - (1-q)\int \varphi^u(v)\,d\mu_L(v) > 0,
  \end{aligned}\end{equation*}
when $ q < 1$. If follows that the pressure $P(q\varphi^u)$ is positive and therefore any}equilibrium state for $q\varphi^u$ must give full measure to the regular set when $q < 1$.

All that remains is to verify that $\mu_L$ is the only equilibrium state for $\varphi^u$ that gives full measure to the regular set.

\begin{proof}[Proof of Theorem \ref{thm:phasetransition}]

Let $\mu$ be an ergodic equilibrium state for $\varphi^u$ supported on $\Reg$. 
Lemma~\ref{l:hyp} tells us that $\mu$ is hyperbolic. Since $P(\varphi^u)=0$ and $\chi^u$ is the only exponent that can be positive, Ruelle's inequality  becomes  the equality
  $$
  h_\mu(\mathcal F) = -\int_{\Reg} \varphi^u(v) \,d\mu(v).
  $$ 
 Ledrappier and Young prove in Theorem A of \cite{LY1} that equality holds in the Ruelle inequality if and only if $\mu$ has absolutely continuous conditional measures on the  Pesin unstable manifolds $W^u(v)$, $v \in T^1M$. 
We can find a set $\Lambda$ with $\mu(\Lambda) > 0$ such that if $v \in \Lambda$, then $v$ is generic for $\mu$ and the Pesin stable and unstable manifolds for $v$ have radius at least $\epsilon$ and depend continuously on $v$. There must be a Pesin unstable manifold $W$ such that $\lambda_W(\Lambda \cap W) > 0$, where $\lambda_W$ is the volume measure on $W$ induced by the Sasaki metric on $T^1M$.  The set 
  $$
  G = \bigcup_{v \in \Lambda \cap W} W^s(v)
  $$
  consists of forward generic points for $\mu$. On the other hand, the transverse absolute continuity of the Pesin stable manifolds tells us that $\mu_L(G) > 0$. Since $\mu_L$ is ergodic, $\mu_L$-a.e.~$v \in G$ is forward generic for $\mu_L$ and therefore  forward generic for both $\mu_L $ and $\mu$. Consequently $\mu = \mu_L$, since all continuous functions must have the same integrals with respect to both measures.
\end{proof}

\subsection*{Acknowledgments} This work was carried out in a workshop at the American Institute of Mathematics. We thank AIM for their support and hospitality. 

\bibliography{bcft-references}{}
\bibliographystyle{plain}
\end{document}